\documentclass{amsart}
\usepackage{amsmath,amsthm,amssymb}
\usepackage[all]{xy}

\newcommand\Z{{\mathbb Z}}
\newcommand\N{{\mathbb N}}

\renewcommand\P{{\mathcal P}}
\newcommand\B{{\mathcal B}}
\newcommand\R{{\mathcal R}}
\newcommand\F{{\mathcal F}}
\newcommand\EF{{\mathcal{E}}}
\newcommand\G{{\mathcal G}}
\newcommand\s{{\sigma}}

\newcommand\Aut{{\mathsf{Aut}}}
\newcommand\expn{{\mathsf{expn}}}
\newcommand\Sym{{\mathsf{S}}}

\theoremstyle{plain}
\newtheorem{theorem}{Theorem}
\newtheorem{proposition}{Proposition}
\newtheorem{lemma}{Lemma}

\theoremstyle{definition}
\newtheorem{definition}{Definition}
\newtheorem{example}{Example}

\begin{document}

\title{Pattern closure of groups of tree automorphisms}

\author{Zoran \v{S}uni\'c}

\thanks{Partially supported by NSF grant DMS-0805932}

\address{Department of Mathematics, Texas A\&M University, College Station, TX 77843-3368, USA}

\begin{abstract}
It is shown that a group defined by forbidding all patterns of size $s+1$ that do not appear in a given self-simialr group of tree automorphisms is the topological closure of a self-similar, countable, regular branch group, branching over its level $s$ stabilizer. 

As an application, it is shown that there are no infinite, finitely constrained, topologically finitely generated groups of binary tree automorphisms defined by forbidden patterns of size two. 
\end{abstract}


\maketitle


\section*{Introduction}

The group of automorphisms of a regular rooted tree carries three structures, namely a self-similarity structure (related to symbolic dynamics on the tree), a metric structure (with Cantor set topology), and a group theoretic structure (of an iterated wreath product). Each of the tree structurers comes with a naturaly associated closure oerator. Namely, given a set $S$ of tree automorphisms, we may consider the self-similar closure of $S$ (the smallest self-similar set containing $S$), the topological closure of $S$, and the group closure of $S$ (the group generated by $S$). The study of the interaction of these three closures naturally leads to the study of patterns in tree automorphisms. 

The main results proved here are as follows. 

\begin{theorem}\label{t:contraction-survives}
Let $G$ be a finitely constrained group of tree automorphisms of $X^*$ defined by allowing all patterns of size $s+1$, $s \geq 0$, that appear in some self-similar group $K$ (and forbidding those that do not). Then $G$ is the topological closure (in $\Aut(X^*)$) of a self-similar, countable, regular branch group $H$, branching over its level $s$ stabilizer $H_s$.

Moreover, if $K$ is contracting, $H$ may be chosen to be contracting as well. 
\end{theorem}

As an application of Theorem~\ref{t:contraction-survives}, we prove the following. 

\begin{theorem}\label{t:no2}
There are no infinite, finitely constrained, topologically finitely generated groups of binary tree automorphisms defined by forbidden patterns of size at most 2.
\end{theorem}

Note that the closure of the first Grigorchuk group is an infinite, finitely constrained, topologically finitely generated group of binary tree automorphisms defined by patterns of size 4~\cite{grigorchuk:unsolved}. The closures of the groups defined by polynomials in~\cite{sunic:hausdorff} provide examples of infinite, finitely constrained, topologically finitely generated groups of binary tree automorphisms defined by patterns of size $s$ (but not size $s-1$), for any $s \geq 4$. Thus, by Theorem~\ref{t:no2}, the question of existence of infinite, finitely constrained, topologically finitely generated groups of binary tree automorphisms defined by patterns of size $s$ remains open only for $s=3$.  

The necessary background on groups of automorphisms of rooted regular trees is provided in the next two sections, which are followed by a section in which the main results are proved. More extensive background infromation may be found in~\cite{grigorchuk:unsolved,nekrashevych:book-self-similar,grigorchuk-s:standrews,bartholdi-g-s:branch}. 


\section{Background on symbolic dynamics on rooted trees}

\subsection{Rooted trees}
Let $X$ be a finite alphabet of cardinality $k$ (our standard choice is $X=\{0,1,\dots,k-1\}$). The rooted tree over $X$ is the $k$-ary rooted tree in which the vertices are the finite words over $X$, the empty word $\emptyset$ is the root and every vertex $u$ is connected by $k$ directed edges to its $k$ children $ux$, for $x$ in $X$. The edge connecting $u$ to $ux$ is labeled by $x$. Level $n$ of the tree $X^*$ is the set $X^n$ of words of length $n$ over $X$. We use $X^*$ to denote both the rooted tree over $X$ and the set of all words over $X$. Note also that the rooted tree $X^*$ is the right
Cayley graph of the free monoid $X^*$ over $X$.

\subsection{Portrait space}

Let $A$ be a finite alphabet (in order to avoid confusion, this alphabet is usually disjoint from $X$). The portrait space on the tree $X^*$ over the alphabet $A$ is the space $A^{X^*}$ of all maps from $X^*$ to $A$.  This space is also called the shift space or the full shift space on $X^*$ over the alphabet $A$. The elements of $A^{X^*}$ are called portraits ($X$-tree portraits over $A$). For a portrait $g$ in the portrait space, denote by $g_{(u)}$ the symbol from $A$ at vertex $u$ in the tree (note that $(u)$ is in the subscript position with respect to $g$). The symbol $g_{(u)}$ is sometimes called the decoration at $u$ in the portrait $g$ and the alphabet $A$ is called the decoration alphabet. 

The portrait space $A^{X^*}$ is a metric space in which, for distinct portraits $g$ and $h$, the distance is given by
\[
 d(g,h) =
 \sup \left\{\frac{1}{2^{|u|}} \mid \ u \in X^*, \ g_{(u)} \neq h_{(u)} \right\}.
\]
The topology on $A^{X^*}$ is just the product topology on $A^{X^*}$ induced by the discrete finite space $A$. Thus, as long as $|A| \geq 2$, $A^{X^*}$ is a Cantor set (in particular, it is compact).

For $u$ in $X^*$, the section map $\s_u:A^{X^*} \to A^{X^*}$ at $u$ (also known as the shift map) on the portrait space is defined by
\[ (\s_u(g))_{(v)} = g_{(uv)}. \]
The section maps provide a right action of $X^*$ on the portrait space by continuous maps. Note that, more generally, portrait spaces may be defined over any semigroup (not only over the free monoid $X^*$ as defined here; see for instance~\cite{coornaert-p:b-symbolic}). 

\begin{definition}
A set of portraits is self-similar if it is invariant under the section maps. Other terms used for self-similar sets are
$X^*$-invariant or shift invariant sets.
\end{definition}

Note that the case $k=1$ is not excluded from our considerations. In this case the tree $X^*$ has the structure of a ray (one-way inifinite path), the monoid $X^*$ is isomorphic to the monoid of natural numbers $\N$ and the shift space $A^\N$ is just the standard one-dimensional one-way shift (see~\cite{lind-m:book-symbolic} or ~\cite{kitchens:book-symbolic}). 

\subsection{Forbidden patterns}

Let $s \geq 1$. Rooted tree of size $s$ over $X$, is the subtree of $X^*$ consisting of the vertices in $X^{[s]} = \cup_{i=0}^{s-1} X^i$ (we denote this subtree also by $X^{[s]}$). An $X$-tree pattern of size $s$ over $A$ is a map in $A^{X^{[s]}}$. All 8 $X$-tree patterns, where $X=\{0,1\}$, of size 2 over $A=\{\square,\blacksquare\}$ are presented in Figure~\ref{f:d4}.
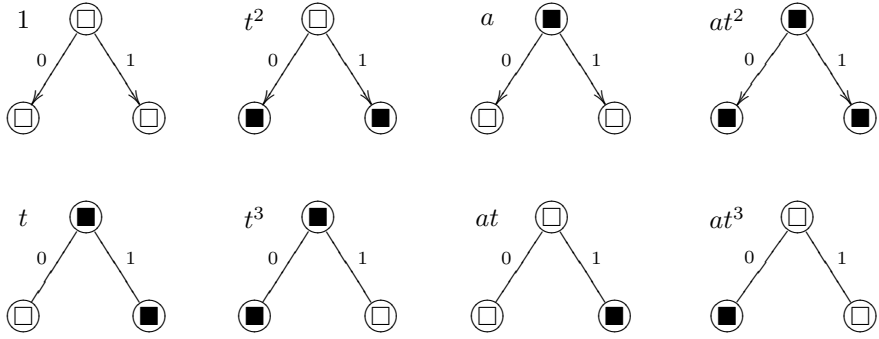
\begin{figure}[!ht]
\[
\xymatrix@C=10pt{
 1& *+[o][F-]{\square} \ar@{->}[dl]_{0}\ar@{->}[dr]^{1}&  &&
 t^2& *+[o][F-]{\square} \ar@{->}[dl]_{0}\ar@{->}[dr]^{1}&  &&
 a& *+[o][F-]{\blacksquare} \ar@{->}[dl]_{0}\ar@{->}[dr]^{1}&  &&
 at^2& *+[o][F-]{\blacksquare} \ar@{->}[dl]_{0}\ar@{->}[dr]^{1}&  &&
 \\
 *+[o][F-]{\square} && *+[o][F-]{\square} &&
 *+[o][F-]{\blacksquare} && *+[o][F-]{\blacksquare} &&
 *+[o][F-]{\square} && *+[o][F-]{\square} &&
 *+[o][F-]{\blacksquare} && *+[o][F-]{\blacksquare}
 \\
 t& *+[o][F-]{\blacksquare} \ar@{-}[dl]_{0}\ar@{-}[dr]^{1} & &&
 t^3& *+[o][F-]{\blacksquare} \ar@{-}[dl]_{0}\ar@{-}[dr]^{1} && &
 at& *+[o][F-]{\square} \ar@{-}[dl]_{0}\ar@{-}[dr]^{1} & &&
 at^3& *+[o][F-]{\square} \ar@{-}_{0}[dl]\ar@{-}[dr]^{1} &
 \\
 *+[o][F-]{\square} && *+[o][F-]{\blacksquare} &&
 *+[o][F-]{\blacksquare} && *+[o][F-]{\square} &&
 *+[o][F-]{\square} && *+[o][F-]{\blacksquare} &&
 *+[o][F-]{\blacksquare} && *+[o][F-]{\square}
}
\]
 \caption{Patterns of size $2$}
 \label{f:d4}
\end{figure}
A tree portrait $g$ contains the tree pattern $p$ of size $s$ at the vertex $u$ if $g_{(uv)} = p_{(v)}$, for $v \in X^{[s]}$.

Let $\F$ be any set of $X$-tree patterns over $A$. Denote by $\G(\F)$ the set of all portraits in the portrait space $A^{X^*}$ that do not contain any pattern from $\F$ at any vertex. A subset $G$ of the portrait space is defined by a set of forbidden patterns if $G=\G(\F)$ for some set of tree patterns $\F$. The set $\F$ is called the set of forbidden tree patterns defining $G$. 

\begin{theorem}\label{when-patterns}
Let $G$ be a set of $X$-tree portraits over $A$. The following are
equivalent.

\text{(i)} $G$ is closed, self-similar subset of the full tree portrait space $A^{X^*}$.

\textup{(ii)} $G$ is defined by a set of forbidden $X$-tree
patterns.
\end{theorem}

Closed self-similar sets of portraits are called portrait subspaces (or sometimes portrait spaces, shifts, or subshifts). A portrait space defined by finitely many forbidden patters is called a portrait space of finite type.

\begin{example}\label{ex:d4}
Let $X=\{0,1\}$, $A=\{\square,\blacksquare\}$ and consider the $X$-tree patterns of size 2 over $A$ provided in Figure~\ref{f:d4}.

If we forbid the patterns in the bottom row, i.e., we define the set of forbidden patterns $\B = \{t,t^3,at,at^3\}$, the automorphisms in the corresponding portrait space of finite type $\G(\B)$ can be characterized as follows. A portrait $g$ belongs to $\G(\B)$ if and only if, for every vertex $u$ in $X^*$,
\[ g_{(u0)} = g_{(u1)}. \]

Similarly, we may forbid the patterns in the right half of Figure~\ref{f:d4}, i.e., we define the set of forbidden patterns $\R=\{a,at^2,at,at^3\}$. A portrait $g$ belongs to the portrait space of finite type $\G(\R)$ if and only if, for every vertex $u$ in $X^*$,
\[ g_{(u)}+g_{(u0)}+g_{(u1)}=0, \]
where we interpret the addition and the equality modulo 2, and we interpret $\square$ as 0 and $\blacksquare$ as 1.
\end{example}


\section{Background on groups of tree automorphisms}

Let $X^*$ be a rooted k-ary tree tree. We consider the special case when the alphabet $A$ is the finite symmetric group $\Sym(X)$. i.e., the case when the decoration at each vertex of the tree is a permutation of the alphabet $X$. Every portrait $g$ of this type defines a rooted tree automorphism of $X^*$, also denoted by $g$, defined by 
\[
 g(x_1x_2 \dots x_n) = g_{(\emptyset)}(x_1) g_{(x_1)}(x_2) \dots g_{(x_1x_2 \dots x_{n-1})}(x_n).   
\]
Conversely, if $g$ is a tree automorphism, it defines a portrait on $X^*$, also denoted by $g$, where the permutation of $X$ at the vertex $u$ is uniquely determined by 
\[
 g_{(u)}(x)=y \iff g(ux)=g(u)y,
\]
for $x$ and $y$ in $X$. 

The group $\Aut(X^*)$ of rooted tree automorphisms of $X^*$ inherits the self-similarity and the metric structure from the $X$-tree portrait space $\Sym(X)^{X^*}$. In particular, $\Aut(X^*)$ is compact and so is each of its closed subgroups. 

Note that $\Aut(X^*)$ has the structure of an iterated permutational wreath product
\[
 \Aut(X^*) \cong \Sym(X) \ltimes (\Aut(X^*))^X = \Sym(X) \wr \Aut(X^*) = \Sym(X) \wr ( \Sym(X) \wr ( \Sym (X) \wr \dots )), 
\]
where the isomorphism $\Aut(X^*) \cong \Sym(X) \ltimes (\Aut(X^*))^X$ is given by 
\[
 g \mapsto g_{(\emptyset)} (g|_0,g|_1,\dots,g|_{k-1}),  
\]
and, for $x \in X$, the automorphism $g|_x$ is just the section $\sigma_x(g)$ of $g$ at $x$. If we identify $\Aut(X^*)$ and $\Sym(X) \ltimes (\Aut(X^*))^X$ under this isomorphism then, for any two automorphims $g$ and $h$, 
\[
 gh = g_{(\emptyset)} (g_0,\dots,g_{k-1}) h_{(\emptyset)} (h_0,\dots,h_{k-1}) = 
 g_{(\emptyset)}h_{(\emptyset)} (g_{h(0)}h_0,\dots,g_{h(k-1)}h_{k-1}). 
\]
We will make use of the equalities 
\[
 (f \cdot g)|_u = f|_{g(u)} \cdot g|_u \qquad\text{and}\qquad (g^{-1})|_u = (g|_{g^{-1}(u)})^{-1}
\]
expressing the sections of products and inverses as products and inverses of appropriate sections. 

For a set of tree automorphisms $S$ we may define the group $\langle S \rangle$ generated by $S$, the (topological) closure $\overline{S}$ of $S$ and the smallest self-similar set $\tilde{S}$ of tree automorphisms containing $S$ (it consist of all sections of all elements in $S$). Further, we can combine these closure operators. For instance, the closure of a group of tree automorphisms is a group and therefore $\overline{\langle S\rangle}$ is the smallest closed group containing $S$. Similarly, the closure of a self-similar set is self-similar and therefore $\overline{\tilde{S}}$ is the smallest closed self-similar set containing $S$. Finally, a group of tree automorphisms generated by a self-similar set is self-similar and therefore $\langle \tilde{S} \rangle$ is the smallest self-similar group containing $S$. The smallest closed self-similar group of tree automorphisms containing $S$ is $\overline{\langle \tilde{S} \rangle}$.

A closed group $G$ of tree automorphisms is topologically finitely generated if $G=\overline{\langle S\rangle}$ for some finite set $S$. A tree automorphism $g$ is a finite-state automorphism (we also say that $g$ is defined by a finite automaton) if $\tilde{g}$ (the set of sections of $g$) is finite. A group $G$ of tree automorphisms is called an automaton group if it is generated by a finite self-similar set, i.e., $G = \langle \tilde{S} \rangle$, where $\tilde {S}$ is finite. A self-similar group $G$ of tree automorphisms is contracting if there exist a finite set $\mathcal N$ of automorphisms such that, for every $g$ in $G$, there exists a level $n$ (depending on $g$) such that, for all $m \geq n$, sections of $g$ at level $m$ are elements of $\mathcal N$. Note that finitely generated contracting groups are automaton groups (each element of a contracting group has only finitely many distinct sections, so it is a finite-state automorphism). 

\begin{proposition}
Let $G=\G(\F)$ be a closed, self-similar subset of the tree portrait space $\Sym(X)^{X^*}$ defined by a set of forbidden patterns $\F$. The set $G$ is a subgroup of $\Aut(X^*)$ if and only if, for every $s \geq 1$, the set of essential patterns $\EF_s$ of size $s$ (patterns of size $s$ that actually appear in some element of $G$) forms a subgroup of $\Aut(X^s)$ (the automorphism group of the finite regular tree over $X$ of depth $s$).

In case $\F$ is a finite set of patterns of size $s$, $G$ is a group if and only if $\EF_s$ is a subgroup of $\Aut(X^s)$.
\end{proposition}

\begin{theorem}
Let $G$ be a group of tree automorphisms of $X^*$. The following are equivalent.

\textup{(i)} The group $G$ is closed, self-similar subgroup of $\Aut(X^*)$.

\textup{(ii)} The group $G$ is defined by a set of forbidden patterns.
\end{theorem}

A group of tree automorphisms defined by a finite set of forbidden patterns is called a finitely constrained group (a more appropriate term would probably be a group of finite type, as in~\cite{grigorchuk:unsolved}, where this kind of groups were introduced, but this term seems to be already overused and so we will avoid it; the term finitely constrained group was used for the first time in~\cite{grigorchuk-n-s:oberwolfach2}).

\begin{example}
Consider again the patterns of size 2 given in Figure~\ref{f:d4} and interpret $\square$ as the trivial permutation of $X=\{0,1\}$ and $\blacksquare$ as the non-trivial permutation $(01)$. The group $\Aut(X^2)$ of tree automorphisms of the $X$-tree of depth 2 is isomorphic to the dihedral group $D_4$ and is generated by the automorphism $t$ of order 4 and the automorphism $a$ of order 2 (subject to the relation
$ata=t^3$). 

Note that, for self-similar groups of binary tree automorphisms, being infinite and being transitive on each level of the tree are equivalent properties (see~\cite{bondarenko-al:classification32-1} or~\cite[Lemma~3, page 112]{bondarenko-al:classification32}).   

The only proper transitive subgroups of $\Aut(X^2)$ are the groups $\{1,t,t^2,t^3\}$, with complement is $\R$, and the group $\{1,a,t^2,at^2\}$, with complement is $\B$. Thus, $\G(\R)$ and $\G(\B)$ are the only infinite, finitely constrained groups of binary tree automorphisms defined by forbidden patterns of size 2 (in addition to the full group $\Aut(X^*)$, which is defined by declaring the empty set to be the set of forbidden patterns). The group $\G(\B)$ appears explicitly in~\cite[page 174]{grigorchuk:unsolved} as one of the simplest nontrivial examples of finitely constrained groups (nontrivial in the sense that the group is neither finite nor $\Aut(X^*)$). 
\end{example}

\subsection{Pattern closure construction}

Given a self-similar group $K$ and we may construct the finitely constrained group $\G(\F_s(K))$ defined by the set of forbidden patterns $\F_s(K)$ of size $s$, which is simply the set of patterns of size $s$ that do not appear in any element of $K$. Theorem~\ref{t:contraction-survives} describes the finitely constrained groups that can be obtained by this pattern closure construction. 

Theorem~\ref{t:contraction-survives} may be seen as a refinement of the direction (ii) implies (i) of Theorem~\ref{t:when-fc} below, since its proof provides an explicit way to construct the group $H$ (and since $H$ is countable). In addition, Theorem~\ref{t:contraction-survives} shows that the contraction property is, in  a sense, compatible with the pattern closure construction. 

\begin{theorem}\label{t:when-fc}
Let $G$ be a group of tree automorphisms of $X^*$ and $s \geq 0$. The
following are equivalent.

\textup{(i)} The group $G$ is the closure of some self-similar, regular branch group $H$, branching over its level $s$ stabilizer $H_s$.

\textup{(ii)} The group $G$ is finitely constrained group defined by patterns of size $s+1$.
\end{theorem}

The direction (ii) implies (i) is proved in~\cite[Proposition~7.5]{grigorchuk:unsolved} and the other direction in~\cite[Theorem~3]{sunic:hausdorff}. Recall that a group $H$ is regular branch group over its level $s$ stabilizer $H_s$ if and only if for all $h_0,\dots,h_{k-1} \in H_s$ the tree automorphism $(h_0,h_1,\dots,h_{k-1})$ is also an element of $H_s$. 

The following example of the pattern closure construction plays a role in the proof of Theorem~\ref{t:no2}. 

\begin{example}\label{ex:odometers}
The group $\G(\R)$ from Example~\ref{ex:d4} is just one example in the family of finitely constrained groups defined by the pattern closure construction with respect to various sizes applied to the, so called, odometer group ($\G(\R))$ corresponds to size 2). 

The $k$-ary odometer automorphism $t$ of $X^*$ is defined by 
\[ t = \rho (1,1,\dots,1,t) \]
where $\rho =(0 \ 1 \ \dots \ k-1)$ is the standard cycle on the alphabet $X=\{0,\dots,k-1\}$. The group $T=\langle t \rangle$ is self-similar, contracting, level transitive group. 

For a fixed size $s+1$, $s \geq 0$, define $G(k,s+1)=\G(\F_{s+1}(T))$ as the finitely constrained group of $k$-ary rooted tree automorphisms for which the forbidden patters are precisely the patterns of size $s+1$ that do not appear in any element of $\langle t \rangle = T \cong \Z$. 
\end{example}


\section{Proofs of Theorem 1 and Theorem 2}

For a word $u$ over $X$ and a tree automorphism $f$, denote by $\delta_u(f)$ the unique tree automorphism that stabilizes level $|u|$ and has trivial section at each vertex at level $|u|$ except at $u$ where its section is equal to $f$. 

\begin{lemma}\label{l:delta-conjugates}
Let $h$ and $g$ be automorphisms of the tree $X^*$. For any vertex $u$, 
\[
 (\delta_u(h))^g = \delta_v(h^{g|_v}), 
\]
where $v=g^{-1}(u)$. 
\end{lemma}

\begin{proof}
Let $|u|=n$ and $v$ be arbitrary vertex at level $n$. Since $\delta_u(h)$ stabilizes level $n$ of $X^*$, we have $g^{-1}\delta_u(h)g(v)= g^{-1}g(v) = v$. Thus $(\delta_u(h))^g$ stabilizes level $n$. Further, 
\begin{align*}
 \delta_u(h))^g|_v 
  &= (g^{-1} \delta_u(h) g)|_v = g^{-1}|_{\delta_u(h)g(v)} \cdot \delta_u(h)|_{g(v)} \cdot g|_v = \\
  &= g^{-1}|_{g(v)} \cdot \delta_u(h)|_{g(v)} \cdot g|_v 
   = (g|_{g^{-1}g(v)})^{-1} \cdot \delta_u(h)|_{g(v)} \cdot g|_v = \\ 
  &= (g|_v)^{-1} \delta_u(h)|_{g(v)} g|_v = (\delta_u(h)|_{g(v)})^{g|_v} = \\ 
  &= \begin{cases}
       h^{g|_v}, & g(v)=u \\
       1, & g(v) \neq u, 
     \end{cases} 
\end{align*}
showing that $(\delta_u(h))^g = \delta_v(h^{g|_v})$, where $v=g^{-1}(u)$. 
\end{proof}

\begin{proof}[Proof of Theorem~\ref{t:contraction-survives}]
Let $\P$ be the set of patterns of size $s+1$ appearing in the elements of the the self-similar group $K$. Because $K$ is self-similar, this is the set of patterns of size $s+1$ appearing at the root in the elements of $K$. Let $S=\{g_1,\dots,g_m\}$ be a set of elements in $K$ such that every pattern in $\P$ appears at the root in at least one of the automorphisms in $S$ (note that the set $\P$ is finite, so $S$ may be chosen to be finite as well).  Let $L=\langle \tilde{S} \rangle$ be the smallest self-similar group containing $S$ (this is a subgroup of $K$) and let $L_s=\langle S' \rangle$ be the stabilizer of level $s$ in $L$. Note that $\tilde{S}$ is countable (since $S$ is finite and every tree  automorphism has no more than countably many sections). Therefore $L$ is countable and so are $L_s$ and $S'$.  

Let 
\[
 D = \{\ \delta_u(h) \mid h \in S', \ u \in X^* \ \}
\]
and
\[
 H = \langle D \cup \tilde{S} \rangle.
\]
Note that $H$ is self-similar. Indeed, $H = \langle D \cup L \rangle$ and all sections of the elements in $D \cup  L$ are trivial or elements in the self-similar group $L$. Therefore $D \cup L$ is a self-similar set and $H$ itself is self-similar. 

We claim that $H_s = \langle D \rangle$. 

Since every element $h \in S'$ stabilizes $s$ levels of the tree $X^*$, $\delta_u(h)$ stabilizes $s+|u|$ levels. Therefore $\langle D \rangle$ is a subgroup of $H_s$. 

Further, by Lemma~\ref{l:delta-conjugates}, for any word $u$, $g \in \tilde{S}$ and $h \in S'$, $(\delta_u(h))^g = \delta_v(h^{g|_v})$, where $v=g^{-1}(u)$. Since $L_s$ is normal in $L$, there exist $h_1,\dots,h_r \in S'$ and expnonents $\epsilon_1,\dots,\epsilon_r$ in $\{-1,1\}$ such that $h^{g|_v} = h_{1}^{\epsilon_1} \dots h_{r}^{\epsilon_r}$. Therefore 
\[
 (\delta_u(h))^g = \delta_v(h^{g|_v}) = \delta_v(h_{1}^{\epsilon_1} \dots h_{r}^{\epsilon_r}) = 
 \delta_v(h_{1})^{\epsilon_1} \dots \delta_v(h_{r})^{\epsilon_r}.
\]  
The last equality shows that the group $\langle D \rangle$ is normal subgroup of $H$. 

Since $\langle D \rangle$ is normal in $H = \langle D \cup \tilde{S} \rangle$, any element of $H_s$ can be written as a product of an element in $\langle D \rangle$ and an element in $\langle \tilde{S} \rangle = L$ stabilizing $s$. But the generators of $L_s$ are in $\langle D \rangle$, which shows that $H_s = \langle D \rangle$. 

The group $H$ is a regular branch group, branching over its stabilizer $H_s$ of level $s$. This is clear since, for any words $u_0,\dots,u_{k-1}$ in $X^*$, and elements $h_0,\dots,h_{k-1}$ in $S'$, 
\begin{align*}
 (\delta_{u_0}(h_{0}),\dots,\delta_{u_{k-1}}(h_{{k-1}})) 
  &= (\delta_{u_0}(h_{0}),1,\dots,1) \cdots (1,\dots,\delta_{u_{k-1}}(h_{{k-1}})) = \\ 
  &= \delta_{0u_0}(h_{0}) \cdots \delta_{(k-1)u_{k-1}}(h_{{k-1}}) \in H_s.
\end{align*}

By Theorem~\ref{t:when-fc}, the closure $\overline{H}$ is a finitely constrained group, defined by patterns of size $s+1$. Moreover, since $H$ is self-similar, the patterns defining $\overline{H}$ are the patterns of size $s+1$ appearing at the root in the elements of $H$. Since, for nonempty words $u$ and $h \in S'$, $\delta_u(h)$ stabilizes level $s+1$, the patterns of size $s+1$ appearing at the root in the elements of $H$ are precisely the patterns of size $s+1$ appearing at the root of the elements in $L$, and these are the patterns defining $G$. Therefore $\overline{H}=G$.

Assume now, in addition, that $K$ is contracting over the finite set $\mathcal N$. Redefine $S$ in the above construction so as to include $\mathcal N$. Then $S$ is finite and, because of the contraction property, so is $\tilde{S}$. Thus $L$ is in this case an automaton group that is contracting over $\mathcal N$. The group $H$ is also contracting over $\mathcal N$ since, for $h \in S'$ each section of $\delta_u(h)$ at level $|u|$ is an element of $L$ and $L$ is contracting over $\mathcal N$.
\end{proof}

Note that in the contracting case $\tilde{S}$ is finite and since $L=\langle \tilde{S} \rangle$ is finitely generated so is its finite index subgroup $L_s$. This means that $S'$ may be chosen to be finite as well. Further, in some situations the set $D = \{ \delta_u(h) \mid h \in S', \ u \in X^* \}$ in the definition of $H$ may be replaced by some subset such as, for instance, $D' = \{ \delta_{0^n}(h) \mid h \in S', \ n=0,1,\dots\}$. 

The claim of Theorem~\ref{t:no2} follows if we prove than none of the groups $\Aut(X^*)$, $\G(\B)$, and $\G(\R)$ is topologically finitely generated. This is known for $\Aut(X^*)$ (see~\cite{grigorchuk:unsolved}), the claim for $\G(\R)$ is proved in more general form in Proposition~\ref{p:odometers}, and the claim for $\G(\B)$ is proved in Proposition~\ref{p:B}. 

\begin{proposition}\label{p:odometers}
The finitely constrained group $G(k,s+1)$ (defined in Example~\ref{ex:odometers} by allowing the patterns of size $s+1$ that appear in the odometer group) is not topologically finitely generated, for $k \geq 2$, $s \geq 0$.
\end{proposition}

\begin{proof}
We first explicitly determine a self-similar, countable, regular branch group $H$, branching over its level stabilizer $H_s$, such that $G$ is the closure of $H$ in $\Aut(X^*)$. In order to accomplish this we follow the argument in the proof of Theorem~\ref{t:contraction-survives} (we follow the argument somewhat loosely, since in the concrete situation some of simplifications, as indicated in the remarks after the proof of Theorem~\ref{t:contraction-survives}, apply).

The role of $L$ may be played by $T$ itself. The stabilizer $T_s$ of level $s$ in $T$ is generated by $t^{k^s}$. Define $t_s = t^{k^s}$ and, for $n \geq s$,
\[ t_{n+1} = (1,1,\dots,t_n). \]

Let $H=\langle t, t_{s+1},\dots \rangle$. Then, for the level $s$ stabilizer in $H$, we have
\[
 H_s = \langle \ t_n^{t^i} \mid \ n \geq s, \ i=0,\dots,k^{n-s}-1 \ \rangle
\]

The closure $\overline{H}$ of $H$ in $\Aut(X^*)$ is precisely $G$. This implies that $G/G_n = H/H_n$, for $n \geq 0$.

Therefore, in order to show that $G$ is not topologically finitely generated, it is sufficient to show that, for $n \geq s+1$, the minimal number of generators of $H_{[n]} = H/H_n$ is $n-s$.

For $n \geq s+1$, let $A_n = C_{k^{s+1}} \times C_k \times \dots \times C_k$, where $C_m$ denotes the standard cyclic group of order $m$ (the elements are the residue classes modulo $m$) and the total number of factors is $n-s$. We claim that, for $n \geq s+1$, there exists a surjective homomorphism from $H_{[n]}$ to $A_n$.

First, since the generators $t_n, t_{n+1}, \dots$ stabilize level $n$, the group $H_{[n]}$ is generated by (the cosets of) $\{t, t_{s+1}, \dots, t_{n-1}\}$. Define a map $\beta_n$ from the set of group words over $\{t,t_{s+1},\dots,t_{n-1} \}$ to $A_n$ by setting
\[ 
 \beta_n(W) = (\expn_t(W),\expn_{t_{s+1}}(W),\dots, \expn_{t_{n-1}}(W)),
\]
where $\expn_{t_*}(W)$ denotes the total expnonent of the letter $t_*$ in $W$. We claim that the map $\beta_n$ represents a surjective homomorphism from $H_{[n]}$ to $A_n$. The surjectivity and the homomorphism property follow trivially, once we show that $\beta_n$ is well defined (as a map from $H_{[n]}$). Therefore, we need to show that, for every group word $W$ over $\{t, t_{s+1}, \dots, t_{n-1}\}$ representing the identity in $H_{[n]}$ (i.e. every group word $W$ representing an element in the stabilizer $H_n$), $\beta_n(W) = (0,0,\dots,0)$. We do this by induction on $n$.

For $n=s+1$, $H_{[s+1]} = \langle t \rangle$ and, since the smallest power of $t$ stabilizing level $s+1$ is $t^{k^{s+1}}$, any group word over $\{t\}$ representing the identity in $H_{[s+1]}$ is a power of $t^{k^{s+1}}$.

Let $n > s+1$ and assume that the inductive claim is true for $n-1$. Let $W$ be a group word over $\{t,t_{s+1},\dots,t_{n-1}\}$ representing the identity in $H_{[n]}$. In particular, the word $W$ must represent an element of the level stabilizer $H_{s+1}$. Since all generators $t_{s+1},t_{s+2},\dots$ stabilize level $s+1$, we conclude that $\expn_t(W)$ must be divisible by $k^{s+1}$. Let $W_0,W_1, \dots, W_{k-1}$ be the group words over $\{t,t_{s+1},\dots,t_{n-2}\}$ obtained by decomposition from the word $W$. Since $W$ represents the identity in $H_{[n]}$ (i.e., it stabilizes level $n$), the words $W_i$, $i=0,\dots,k-1$, represent the identity in $H_{[n-1]}$ (i.e., they stabilize level $n-1$). We
have
\begin{align}
 \expn_{t}(W_0) + \dots + \expn_{t}(W_{k-1}) &= \expn_{t}(W) +
 k^s\expn_{t_{s+1}}(W), \notag \\
 \expn_{t_{s+1}}(W_0) + \dots + \expn_{t_{s+1}}(W_{k-1}) &= \expn_{t_{s+2}}(W), \notag \\
 \dots \label{e:expnonents} \\
 \expn_{t_{n-2}}(W_0) + \dots + \expn_{t_{n-2}}(W_{k-1}) &=
 \expn_{t_{n-1}}(W). \notag
\end{align}
By the induction hypothesis, for $i=0,\dots,k-1$, $\expn_t(W_i)$ is divisible by $k^{s+1}$, while $\expn_{t_j}(W_i)$ is divisible by $k$, for $j=s+1,\dots,n-2$. Since $\expn_t(W)$ is also divisible by $k^{s+1}$ we conclude from the first equality in~\eqref{e:expnonents} that $\expn_{t_{s+1}}(W)$ is divisible by $k$. The other equalities in~\eqref{e:expnonents} imply that $\expn_{t_{s+2}}(W)$, \dots, $\expn_{t_{n-1}}(W)$ are divisible by $k$.

Since, for $n \geq s+1$, the abelian group $A_n$ has rank $n-s$ and $\beta_n: H_{[n]} \to A_n$ is a surjective homomorphism, we conclude that the closure $\overline{H}=G$ is not topologically finitely generated.
\end{proof}

\begin{proposition}\label{p:B}
The finitely constrained group $G(\B)$ (defined in Example~\ref{ex:d4}) is not topologically finitely generated.
\end{proposition}

\begin{proof}
The proof follows the general outline of the Proof of Proposition~\ref{p:odometers}.

The role of a self-similar, countable, regular branch group $H$, branching over its first level stabilizer $H_1$, such that $G=\G(\B)$ is the closure of $H$ in $\Aut(X^*)$ is played by $H=\langle a,a_1,a_2,a_3,dots \rangle$ (and the role of $L$ by $\langle a,a_1 \rangle$), where 
\[ a = (01) (1,1) , \qquad\qquad a_1 = (a,a) \]
and for $n \geq 2$,
\[ a_{n+1} = (1,a_n). \]
Every generator of $H$ has order 2.

The group $H_{[n]} = H/H_n$, $n \geq 1$, is generated by (the cosets of) $\{a,a_1,a_2,\dots,a_{n-1}\}$. The map $\beta_n: H_{[n]} \to A_n$, where $A_n = C_2^n$, defined by
\[
 \beta_n(W) = (\expn_a(W), \expn_{a_1}(W),\dots,\expn_{a_{n-1}}(W))
\]
is a surjective homomorphism, which shows that $\overline{H} = \G(\B)$ is not topologically finitely generated. 

Indeed, to show that $\beta_n$, for $n \geq 1$, is a surjective homomorphism it suffices to show that it is well defined, i.e., it suffices to show that for a group word $W$ over $\{a,a_1,\dots,a_{n-1}\}$ representing an element in $H_n$ the exponent $\expn_a(W)$ and the exponents $\expn_{a_i}(W)$, $i=1,\dots,n-1$, are even. This can be accomplished by induction on $n$. The claim is clear for $n=1$, since $a$ is the only generator that does not stabilize level 1. In fact, $\exp_a(W)$ must be even for any group word over $\{a,a_1,a_2,\dots\}$ stabilizing at least one level of the tree. Assume that $n \geq 2$ and the claim is correct for $n-1$. For any group word $W$ over ${a,a_1,\dots,a_{n-1}}$ representing an element in $H_n$, let the words $W_0$ and $W_1$ be the group words over $\{a,a_1,\dots,a_{n-1}\}$ obtained by decomposition. These words represent elements in $H_{n-1}$ and the induction hypothesis applies. Since 
\begin{align}
 \expn_{a}(W_0) &= \expn_{a_1}(W), \notag \\
 \expn_{a_1}(W_0) + \expn_{a_1}(W_1) &= \expn_{a_2}(W), \notag \\
 \dots \label{e:expnonents2} \\
 \expn_{a_{n-2}}(W_0) + \expn_{a_{n-2}}(W_1) &= \expn_{a_{n-1}}(W). \notag
\end{align}
and all exponents on the left are even, all exponent on the right are even as well, completing the proof. 
\end{proof}

\section{A remark on level transitivity}

We observed in Example~\ref{ex:d4} that only the transitive subgroups od $\Aut(X^2)$ may lead to infinite (and level transitive) finitely constrained subgroups of $\Aut(X^*)$. Here we provide an example that shows that not all transitive subgroups of $\Aut(X^s)$ yield spherically transitive (or even nontrivial) finitely constrained groups.

\begin{example}
Let $X=\{0,1\}$. Forbid all patterns of size 3 that contain a pattern from $\R$ as the sub-pattern at the root and all patterns that contain a pattern from $\B$ in any of the two bottom sub-patterns of size 2. Let $\F$ be the set of forbidden patterns of size 3 we just defined. There are 16 allowed patterns (the sub-pattern of size 2 at the top comes from $\langle t \rangle$ and the two sub-patterns in the bottom come from $\langle a,t^2 \rangle$). The group of allowed patterns (isomorphic to $C_4 \ltimes (C_2 \times C_2)$) acts transitively on $X^3$. 

A binary tree automorphism $g$ belongs to the finitely constrained group $\G(\F)$ if and only if, for every word $u$ over $X$ and every letter $x$ in $X$,
\[
 g_{(u0)} + g_{(u1)}+g_{(u)}=0, \qquad\text{and}\qquad
 g_{(ux0)} = g_{(ux1)},
\]
where the equalities are considered modulo 2, the trivial permutation on $X$ is regarded as 0 and the non-trivial as 1. It is easy to check that only the trivial automorphism satisfies the above requirements, i.e., $\G(\F) = 1$. Indeed, assume that, for some $u \in X^*$,  $g_{(u)}=1$. Then exactly one of $g_{(u0)}$ or $g_{(u1)}$ must be equal to 1. Without loss of generality, assume $g_{(u0})=1$. The conditions $g_{(u0})=1$, $g_{(u00)}+g_{(u01)}+g_{(u0)}=0$ and $g_{(u00)} = g_{(u01)}$ contradict each other.
\end{example}

\newcommand{\etalchar}[1]{$^{#1}$}
\def\cprime{$'$}


\end{document}